\documentclass[reqno, 12pt]{amsart}

\pdfoutput=1

\usepackage{enumerate}
\usepackage{latexsym}
\usepackage[centertags]{amsmath}
\usepackage{amsfonts}
\usepackage{amssymb}
\usepackage{amsthm}
\usepackage{newlfont}
\usepackage{graphics}
\usepackage{color}
\usepackage{float}
\usepackage{mathrsfs}
\usepackage{diagbox}
\usepackage{hyperref}
\usepackage{longtable}
\usepackage{rotating}
\usepackage{multirow}
\usepackage{extarrows}
\usepackage{cite}
\usepackage[sort,compress,numbers]{natbib}

\usepackage[utf8]{inputenc}

\newtheorem{thm}{Theorem}[section]
\newtheorem{cor}[thm]{Corollary}
\newtheorem{lem}[thm]{Lemma}
\newtheorem{prop}[thm]{Proposition}

\newtheorem{cnj}[thm]{Conjecture}

\theoremstyle{mydefinition}

\theoremstyle{myremark}

\newtheorem{exa}[thm]{Example}

\allowdisplaybreaks[4]

\title{Proof of a conjecture on graph polytope}

\author{Feihu Liu}

\address{School of Mathematical Sciences, Capital Normal University, Beijing 100048, PR China}
\email{\texttt{liufeihu7476@163.com}}
\date{\today}

\begin{document}
\maketitle

\begin{abstract}
Graph polytopes arising from vertex-weighted graphs were first introduced by B\'ona, Ju, and Yoshida.
We prove a conjecture stating that for any simple connected graph, the numerator polynomial of the Ehrhart series of its graph polytope is palindromic, using Stanley's reciprocity theorem.
Furthermore, we introduce hypergraph polytopes and establish that every simple, connected, unimodular hypergraph polytope is an integer polytope.
Additionally, for simple connected uniform hypergraph polytopes, we demonstrate that the numerator polynomial of their Ehrhart series is palindromic.
\end{abstract}

\def\D{{\mathcal{D}}}

\noindent
\begin{small}
 \emph{Mathematic subject classification}: Primary 05A15; Secondary 52B05, 05C22, 05C65.
\end{small}

\noindent
\begin{small}
\end{small}

\section{Introduction}
In this paper, we consistently denote by $\mathbb{R}$, $\mathbb{Q}$, $\mathbb{Z}$, $\mathbb{N}$, and $\mathbb{P}$ the sets of real numbers, rational numbers, integers, nonnegative integers, and positive integers, respectively.

A \emph{graph} $G = (V, E)$ consists of a vertex set $V$ and an edge set $E$, where each edge is an unordered pair of distinct vertices. Throughout, we assume $G$ is simple (no loops or multiple edges). For a graph with $k$ vertices, we denote the vertices of $G$ by $v_1, v_2, \ldots, v_k$; that is, the vertex set is $V = \{v_1, v_2, \ldots, v_k\}$. When vertices $v_i$ and $v_j$ are adjacent, we write $v_iv_j \in E$ (an abbreviation for $\{v_i, v_j\}\in E$) for the corresponding edge.

Given a nonnegative integer $t$ and a simple graph $G = (V, E)$, B\'ona, Ju, and Yoshida~\cite{BonaJu07} enumerated the set
\[
W(G,t) = \left\{ (n_1,\ldots,n_k) \Big| \begin{array}{l}
n_i \in \{0,1,\ldots, t\} \text{ for all } 1 \leq i \leq k, \\
\text{and } n_i + n_j \leq t \text{ if } v_iv_j \in E
\end{array} \right\}.
\]
Elements of $W(G,t)$ are called \emph{vertex weighted graphs}. The real solution set defined by these inequalities forms the $t$-th dilation of a rational convex polytope determined by $G$. Specifically, the \emph{graph polytope} $P(G)$ is defined as
\[
P(G) = \left\{ (n_1,\ldots,n_k) \Big| \begin{array}{l}
0 \leq n_i \leq 1 \text{ for all } 1 \leq i \leq k, \\
\text{and } n_i + n_j \leq 1 \text{ if } v_iv_j \in E
\end{array} \right\}.
\]
Computational results on volumes of graph polytopes for various graph families were obtained by Ju, Kim, and Lee~\cite{JuKim18}.

We now recall fundamental concepts regarding polytopes. A \emph{polytope} $\mathcal{P} \subset \mathbb{R}^d$ is the convex hull of finitely many points. Its \emph{dimension}, denoted $\dim(\mathcal{P})$, is the dimension of its affine hull. A polytope is \emph{rational} if all vertices have rational coordinates, and \emph{integral} if all vertices have integer coordinates.

For a polytope $\mathcal{P} \subseteq \mathbb{R}^d$ and $t \in \mathbb{N}$, the \emph{Ehrhart function} $\mathrm{ehr}(\mathcal{P}, t)$ counts integer points in its $t$-th dilation:
\[
\mathrm{ehr}(\mathcal{P},t) = \left| t\mathcal{P} \cap \mathbb{Z}^d \right|, \quad \text{where} \quad t\mathcal{P} = \{ t \boldsymbol{\alpha} \mid \boldsymbol{\alpha} \in \mathcal{P} \}.
\]
The associated generating function
\[
\mathrm{Ehr}(\mathcal{P},x) = \sum_{t = 0}^{\infty} \mathrm{ehr}(\mathcal{P}, t)  x^t
\]
is the \emph{Ehrhart series} of $\mathcal{P}$. For the graph polytope $P(G)$, it follows immediately from the definitions that $\mathrm{ehr}(P(G), t) = |W(G, t)|$, i.e., the cardinality of $W(G,t)$.

When $\mathcal{P}$ is an integral polytope, Ehrhart~\cite{Ehrhart62} established that $\mathrm{ehr}(\mathcal{P},t)$ is a polynomial in $t$ of degree $d$, known as the \emph{Ehrhart polynomial} of $\mathcal{P}$. In this case, the Ehrhart series takes the form
\[
\mathrm{Ehr}(\mathcal{P},x) = \frac{h^{*}(x)}{(1-x)^{d+1}},
\]
where $\dim\mathcal{P}=d$ and $h^{*}(x)$ is a polynomial of degree at most $d$ in $x$.
The polynomial $h^{*}(x)$ is commonly called the \emph{$h^{*}$-polynomial} of $\mathcal{P}$.
Stanley~\cite{StanleyADM80} demonstrated that the coefficients of $h^{*}(x)$ are always nonnegative integers.

If $\mathcal{P}$ is a rational polytope of dimension $d$, then
\[
\mathrm{ehr}(\mathcal{P},t) = c_d(t)t^d + c_{d-1}(t)t^{d-1} + \cdots + c_1(t)t + c_0(t),
\]
where the $c_i$ are periodic functions $\mathbb{Z}\rightarrow \mathbb{Q}$, forming the \emph{Ehrhart quasi-polynomial} of $\mathcal{P}$.

For $\gamma \in \mathbb{Q}^m$, define $\text{den}(\gamma)$ (the \emph{denominator} of $\gamma$) as the smallest positive integer $q$ such that $q\gamma \in \mathbb{Z}^m$. For $\gamma \in \mathbb{Q}$, this coincides with the denominator of $\gamma$ in lowest terms. For example, if $\gamma = (2, \frac{1}{3}, 0, -\frac{3}{4}, \frac{5}{6}) \in \mathbb{Q}^5$, then $\text{den}(\gamma) = 12$.

\begin{lem}\cite{Ehrhart62}\label{Ehrhartpoled+1}
Let $\mathcal{P}$ be a rational polytope of dimension $d$ in $\mathbb{R}^m$ with vertex set $\widehat{V}$. The Ehrhart series $\mathrm{Ehr}(\mathcal{P},x)$ is a rational function $\frac{R(x)}{T(x)}$ where $R(x)$ and $T(x)$ are polynomials satisfying:
\begin{enumerate}
    \item $T(x) = \prod_{\gamma \in \widehat{V}} (1 - x^{\text{den}\gamma})$ for some representation,
    \item When written in lowest terms, $\deg(R(x)) < \deg(T(x))$,
    \item $x=1$ is a pole of order exactly $d+1$,
    \item No pole has order exceeding $d+1$.
\end{enumerate}
\end{lem}
In general, $T(x) = \prod_{\gamma \in \widehat{V}} (1 - x^{\text{den}\gamma})$ is not the denominator of least degree for $\mathrm{Ehr}(\mathcal{P},x)$. However, the denominator of least degree contains a factor $(1-x)^{d+1}$ but not $(1-x)^{d+2}$, whereas $T(x)$ contains a factor $(1-x)^{|\widehat{V}|}$. Lemma~\ref{Ehrhartpoled+1} is discussed comprehensively in \cite[Theorem 4.6.8]{RP.Stanley}.

For further background on polytopes, see \cite{BeckRobins}, \cite[Chapter 4]{RP.Stanley}, and \cite{Ziegler}.

This paper resolves the following conjecture by Ju~\cite{Ju07}:
\begin{cnj}\cite[Section 4]{Ju07}\label{ConjeConnGrap1}
Let $G$ be a finite simple connected graph. The numerator polynomial of the Ehrhart series
\[
\mathrm{Ehr}(P(G),x) = \sum_{t \geq 0} |W(G,t)| x^t
\]
is symmetric (also called palindromic).
\end{cnj}
The conjecture was confirmed for bipartite graphs (graphs without odd cycles) in \cite[Theorem 3.6]{BonaJu07}.

We also confirm the following conjecture using a result from \cite{XinZhongZhou}.
\begin{cnj}\cite[Remark 6.2]{BonaJuAnn06}\label{ConjeCircularGrap3}
Let $G$ be the graph with vertex set $V=\{1,2,\ldots, k\}$ and edge set consisting of all pairs $\{i,i+1\}$, where $k+1$ is identified with $1$. This graph $G$ is called a \emph{circular graph}.
If $\ell \in \mathbb{P}$ and $k=2\ell$, then the rational function $\mathrm{Ehr}(P(G),x)(1-x)^{2\ell+1}$ is a polynomial in $x$ of degree $2\ell-2$ with symmetric coefficients.
If $\ell \in \mathbb{P}$ and $k=2\ell+1$, then the rational function $\mathrm{Ehr}(P(G),x)(1-x)^{2\ell+2}(1+x)$ is a polynomial in $x$ of degree $2\ell$ with symmetric coefficients.
\end{cnj}
The case of line graphs (i.e., circular graphs without the edge $\{1, k\}$) is treated comprehensively in \cite{XinZhong23}.

We extend the definition of the graph polytope to hypergraphs $H=(V,E^*)$, denoting the hypergraph polytope by $P(H)$. The following results are established (see Section~\ref{Section33} for definitions).
\begin{thm}\label{UnimodularHyperG}
Let $H=(V,E^*)$ be a finite simple connected hypergraph with vertex set $V=\{v_1,v_2,\ldots, v_k\}$ and hyperedge sequence $E^*=(e_1,e_2,\ldots,e_r)$ such that $\cup_{i=1}^r e_i=V$. If $H$ is unimodular, then $P(H)$ is an integer polytope. Moreover, the denominator of $\mathrm{Ehr}(P(H),x)$ is $(1-x)^{k+1}$.
\end{thm}

\begin{thm}\label{UniformHyperGr}
Let $H=(V,E^*)$ be a finite simple connected hypergraph with vertex set $V=\{v_1,v_2,\ldots, v_k\}$ and hyperedge sequence $E^*=(e_1,e_2,\ldots,e_r)$ such that $\cup_{i=1}^r e_i=V$. If $H$ is $s$-uniform, then $\mathrm{Ehr}(P(H),x)$ is a rational function that, when written in lowest terms, takes the form $\frac{M(x)}{Q(x)}$, where $Q(x)$ is a polynomial and $M(x)$ is a symmetric polynomial of degree $\deg Q(x) - (s+1)$.
\end{thm}

This paper is organized as follows.
In Section 2, we prove Conjectures \ref{ConjeConnGrap1} and \ref{ConjeCircularGrap3} using Stanley's reciprocity theorem.
In Section 3, we establish Theorems \ref{UnimodularHyperG} and \ref{UniformHyperGr}.

\section{Graph Polytope}
To establish our results, we introduce Stanley's reciprocity theorem \cite[Theorem 4.1]{StanleyMagic73}.
Let $\Phi$ be an $r \times m$ integer matrix of rank $r$, and define $\textbf{nul}(\Phi)$ as the solution space of the homogeneous linear Diophantine system $\Phi \mathbf{x} = \mathbf{0}$, that is, $\textbf{nul}(\Phi) = \{\boldsymbol{\alpha} \in \mathbb{R}^m \mid \Phi \boldsymbol{\alpha} = 0\}$.
The set $\textbf{nul}(\Phi) \cap \mathbb{N}^m$ forms an additive semigroup with identity.

We define the generating functions $F(\mathbf{x}) = F(x_1, x_2, \ldots, x_m)$ and $\overline{F}(\mathbf{x}) = \overline{F}(x_1, x_2, \ldots, x_m)$ by
\[
F(\mathbf{x}) = \sum_{(\alpha_1, \ldots, \alpha_m) \in \textbf{nul}(\Phi) \cap \mathbb{N}^m} x_1^{\alpha_1} x_2^{\alpha_2} \cdots x_m^{\alpha_m}
\]
and
\[
\overline{F}(\mathbf{x}) = \sum_{(\alpha_1, \ldots, \alpha_m) \in \textbf{nul}(\Phi) \cap \mathbb{P}^m} x_1^{\alpha_1} x_2^{\alpha_2} \cdots x_m^{\alpha_m}.
\]

\begin{thm}[Stanley's reciprocity theorem, \cite{StanleyMagic73}]\label{StanleyReciprocity}
Let $\Phi$ be an $r \times m$ integer matrix of full rank $r$. If there exists a vector $\boldsymbol{\alpha} \in \mathbf{nul}(\Phi) \cap \mathbb{P}^m$, then $F(\mathbf{x})$ and $\overline{F}(\mathbf{x})$ are rational functions in the variables $x_1, \dots, x_m$ satisfying
\[
\overline{F}(x_1, x_2, \ldots, x_m) = (-1)^{m-r} F(x_1^{-1}, x_2^{-1}, \ldots, x_m^{-1}).
\]
\end{thm}

For further references on this reciprocity theorem, we refer to \cite{StanleyReci74,XinReci07}.
Regarding graph polytopes $P(G)$, B\'ona, Ju, and Yoshida \cite{BonaJu07} established the following result.

\begin{thm}\cite{BonaJu07}\label{GraphPolyBonaJu}
Let $G = (V, E)$ be a finite simple connected graph with $|V| = k$. Then there exists a nonnegative integer $s$ such that the Ehrhart series of the graph polytope $P(G)$ satisfies
\begin{align}\label{Ehrhart-Series-PG}
\mathrm{Ehr}(P(G), x) = \frac{H(x)}{(1 - x^2)^s (1 - x)^{k + 1 - s}},
\end{align}
where $H(x)$ is a polynomial of degree at most $k + s$. Moreover, if $G$ is bipartite, then $s = 0$ and $H(x)$ is a symmetric polynomial of degree at most $k$.
\end{thm}

We now establish the following result.
\begin{thm}\label{GraphPolytSymmetric}
Let $G=(V,E)$ be a finite simple connected graph with $|V|=k$.
The Ehrhart series of $P(G)$ is given by Equation \eqref{Ehrhart-Series-PG}.
Then its numerator $H(x)$ is a symmetric polynomial of degree $k+s-2$.
Furthermore, if $G$ is bipartite, then $H(x)$ is symmetric polynomial of degree $k-2$.
\end{thm}
\begin{proof}
Let $|E|=r$. The graph polytope definition yields the homogeneous system:
\begin{align*}
n_i + n_j \leq t \quad &\text{for each edge } e_\ell=v_iv_j \in E, 1 \leq \ell \leq r, \\
\Longleftrightarrow \quad n_i + n_j + b_{\ell} - t = 0 \quad &\text{with } b_{\ell} \geq 0, \ 1 \leq \ell \leq r, \ \text{for } e_\ell=v_iv_j \in E.
\end{align*}
Denote this by $\Phi \boldsymbol{x} = \boldsymbol{0}$, where $\Phi$ is an $r \times (k+r+1)$ matrix of full rank $r$. Define generating functions
\[
F(\boldsymbol{x}) = \sum_{(n_1,\ldots,n_k,b_1,\ldots,b_r,t) \in \mathbf{nul}(\Phi) \cap \mathbb{N}^{k+r+1}} y_1^{n_1} \cdots y_k^{n_k} z_1^{b_1} \cdots z_r^{b_r} x^t
\]
and
\[
\overline{F}(\boldsymbol{x}) = \sum_{(n_1,\ldots,n_k,b_1,\ldots,b_r,t) \in \mathbf{nul}(\Phi) \cap \mathbb{P}^{k+r+1}} y_1^{n_1} \cdots y_k^{n_k} z_1^{b_1} \cdots z_r^{b_r} x^t.
\]

By construction, $\mathrm{Ehr}(P(G),x) = F(\boldsymbol{1},x)$, where $\boldsymbol{1}$ is the all-ones vector. Observe that the vector $\boldsymbol{\beta} = (\boldsymbol{1}, \boldsymbol{1}, 3) \in \mathbf{nul}(\Phi) \cap \mathbb{P}^{k+r+1}$ (where the first $\boldsymbol{1}$ has length $k$, the second length $r$). This implies
\[
\boldsymbol{\alpha} \in \mathbf{nul}(\Phi) \cap \mathbb{P}^{k+r+1} \iff \boldsymbol{\alpha} - \boldsymbol{\beta} \in \mathbf{nul}(\Phi) \cap \mathbb{N}^{k+r+1},
\]
and consequently,
\[
\overline{F}(\boldsymbol{x}) = \left( \prod_{i=1}^k y_i \right) \left( \prod_{j=1}^r z_j \right) x^3 F(\boldsymbol{x}).
\]
Stanley's reciprocity theorem \ref{StanleyReciprocity} gives
\[
x^3 F(\boldsymbol{1},x) = \overline{F}(\boldsymbol{1},x) = (-1)^{k+1} F(\boldsymbol{1},x^{-1}),
\]
which is equivalent to
\[
x^3 \mathrm{Ehr}(P(G),x) = (-1)^{k+1} \mathrm{Ehr}(P(G),x^{-1}).
\]
This implies $H(x) = x^{k+s-2} H(x^{-1})$, proving $H(x)$ is symmetric of degree $k+s-2$. For bipartite $G$, $H(x)$ is symmetric of degree $k-2$ since $s=0$. This completes the proof.
\end{proof}

As an application of Theorem \ref{GraphPolytSymmetric}, we verify Conjecture \ref{ConjeConnGrap1}. For Conjecture \ref{ConjeCircularGrap3}: when $\ell \in \mathbb{P}$ and $k=2\ell$, the result follows directly from Theorem \ref{GraphPolytSymmetric}; when $k=2\ell+1$, the parameter $s=1$ in $\mathrm{Ehr}(P(G),x)$ is confirmed by \cite{XinZhongZhou}. Thus Conjecture \ref{ConjeCircularGrap3} holds.

As a further application, we establish the following result, which was originally conjectured by Bona and Ju~\cite[Section 2.6]{BonaJu07}.

\begin{prop}[First proof in \cite{LeeJu15}]\label{ConjeHyperGrap2}
Let $G$ be the graph of the $d$-dimensional hypercube; that is, its vertex set consists of the $2^d$ points in $\{0,1\}^d$, where two vertices are adjacent if they differ in exactly one coordinate.
Then the numerator polynomial of the Ehrhart series $\mathrm{Ehr}(P(G),x)$ is symmetric and has degree $2^d - 2$.
\end{prop}

Since the $d$-dimensional hypercube is bipartite, this result follows directly from \cite[Theorem 5]{LeeJu15}. Therefore, the first proof of Proposition~\ref{ConjeHyperGrap2} was given by Lee and Ju.

For completeness, we recall that a polynomial $f(x)=a_nx^n+a_{n-1}x^{n-1}+\cdots+a_1x+a_0$ with real coefficients is \emph{unimodal} if there exists an index $0 \leq i \leq n$ satisfying $a_0\leq \cdots \leq a_{i-1}\leq a_i\geq a_{i+1} \geq \cdots \geq a_n$.

We now present an example demonstrating that the numerator polynomial $H(x)$ is not necessarily unimodal. This example appears in \cite[Theorem 2.5]{BonaJu07}.

\begin{exa}
A graph is \emph{complete} if every pair of distinct vertices is adjacent. The complete graph on $k$ vertices is denoted by $G_k$. For small $k$, we compute the Ehrhart series $\mathrm{Ehr}(P(G_k),x)$ using the software \emph{\texttt{LattE}} \cite{DeLoera04} and \emph{\texttt{CTEuclid}} \cite{Xin15}:
{\small
\begin{align*}
\mathrm{Ehr}(P(G_5),x) &= \frac{1+3x+19x^2+14x^3+19x^4+3x^5+x^6}{(1-x)^3(1-x^2)^3}, \\
\mathrm{Ehr}(P(G_6),x) &= \frac{1+4x+48x^2+56x^3+142x^4+56x^5+48x^6+4x^7+x^8}{(1-x)^3(1-x^2)^4}, \\
\mathrm{Ehr}(P(G_7),x) &= \frac{1+5x+109x^2+176x^3+730x^4+478x^5+730x^6+176x^7+109x^8+5x^9+x^{10}}{(1-x)^3(1-x^2)^5}.
\end{align*}}
The numerator polynomials are symmetric for all computed $k$, but fail to be unimodal for $k=5$ and $k=7$.
\end{exa}

For further studies of Ehrhart quasi-polynomials associated with graph polytopes, including linear graphs, circular graphs, star graphs, cubic graphs, and complete bipartite graphs, we refer to \cite{BonaJu07,Ju07,XinZhong23}.

\section{Hypergraph Polytope}\label{Section33}

A \emph{hypergraph} $H = (V, E^* = (e_i)_{i \in I})$ consists of a finite vertex set $V$ and a family of subsets $E^* = (e_i)_{i \in I}$ of $V$ called \emph{hyperedges}, where $I$ is a finite index set. A hyperedge $e \in E^*$ with $|e| = 1$ is a \emph{loop}. The hypergraph $H$ is \emph{without repeated hyperedges} if $i \neq j$ implies $e_i \neq e_j$. It is \emph{simple} if $e_i \subseteq e_j$ implies $i = j$; this condition implies no repeated hyperedges. The hypergraph is \emph{connected} if any two distinct vertices are connected by a path, and \emph{$s$-uniform} if every hyperedge contains exactly $s$ vertices. For a comprehensive introduction, see \cite{Bretto}.

Let $V = \{v_1, v_2, \dots, v_k\}$ and $E^* = (e_1, e_2, \dots, e_r)$ with $\bigcup_{i=1}^r e_i = V$. The \emph{incidence matrix} of $H = (V, E^*)$ is the $r \times k$ matrix $A = (a_{ij})$, where $a_{ij} = 1$ if $v_j \in e_i$, and $a_{ij} = 0$ otherwise. A matrix is \emph{totally unimodular} if every square submatrix has determinant $-1$, $0$, or $1$; consequently, all entries lie in $\{-1, 0, 1\}$. The hypergraph $H$ is \emph{unimodular} if its incidence matrix is totally unimodular.

We now extend the definition of a graph polytope to hypergraphs. The \emph{hypergraph polytope} $P(H)$ for a finite simple hypergraph $H = (V, E^*)$ is defined as
\[
P(H) = \left\{ \mathbf{n} \in \mathbb{R}^k \mid \mathbf{n} \geq \mathbf{0},\  A\cdot \mathbf{n} \leq \mathbf{1} \right\},
\]
where $A$ is the incidence matrix of $H$ and $\mathbf{n} = (n_1, n_2, \dots, n_k)^T$.

Hypergraph polytopes are not necessarily integer polytopes. However, under unimodularity, Theorem \ref{UnimodularHyperG} holds.
\begin{proof}[Proof of Theorem \ref{UnimodularHyperG}]
By definition, the incidence matrix $A$ of a unimodular hypergraph is totally unimodular. When $A$ is totally unimodular and $\mathbf{1}$ is integral, all vertices of the polytope $\{\mathbf{n} \in \mathbb{R}^k \mid \mathbf{n} \geq \mathbf{0},\ A\cdot\mathbf{n} \leq \mathbf{1}\}$ are integral (see \cite[Chapter 19]{Schrijver} or \cite{Baldoni04}). Thus $P(H)$ is an integer polytope. Furthermore, $\dim P(H) = k$, the number of vertices of $H$.
\end{proof}

Let us now prove Theorem~\ref{UniformHyperGr}.
\begin{proof}[Proof of Theorem~\ref{UniformHyperGr}]
By the definition of the hypergraph polytope, we obtain the system:
\begin{align*}
&n_{i_1} + n_{i_2} + \cdots + n_{i_s} \leq t, \quad \text{if  } e_\ell=\{v_{i_1}, v_{i_2}, \ldots, v_{i_s}\} \in E,\quad 1 \leq \ell \leq r, \\
\Longleftrightarrow \quad & n_{i_1} + n_{i_2} + \cdots + n_{i_s} + b_{\ell} - t = 0,\ \text{where } b_{\ell} \geq 0,\ 1 \leq \ell \leq r,\ \text{if }e_\ell=\{v_{i_1}, v_{i_2}, \ldots, v_{i_s}\} \in E.
\end{align*}
Denote this homogeneous linear Diophantine system by $\Phi \mathbf{x} = \mathbf{0}$, where $\Phi$ is an $r \times (k + r + 1)$ matrix of full rank $r$. The generating functions $F(\mathbf{x})$ and $\overline{F}(\mathbf{x})$ are defined as
\[
F(\mathbf{x}) = \sum_{(n_1,\ldots,n_k, b_1,\ldots,b_r,t) \in \mathbf{nul}(\Phi) \cap \; \mathbb{N}^{k+r+1}} y_1^{n_1} \cdots y_k^{n_k} z_1^{b_1} \cdots z_r^{b_r} x^t
\]
and
\[
\overline{F}(\mathbf{x}) = \sum_{(n_1,\ldots,n_k, b_1,\ldots,b_r,t) \in \mathbf{nul}(\Phi) \cap \; \mathbb{P}^{k+r+1}} y_1^{n_1} \cdots y_k^{n_k} z_1^{b_1} \cdots z_r^{b_r} x^t.
\]

Note that $\mathrm{Ehr}(P(H), x) = F(\mathbf{1}, x)$. Observe that the vector $\boldsymbol{\beta} = (\boldsymbol{1}, \boldsymbol{1}, s+1)$ lies in $\mathbf{nul}(\Phi) \cap \mathbb{P}^{k+r+1}$ (where the first $\boldsymbol{1}$ has length $k$, the second length $r$). Thus,
\[
\boldsymbol{\alpha} \in \mathbf{nul}(\Phi) \cap \mathbb{P}^{k+r+1} \quad \Longleftrightarrow \quad \boldsymbol{\alpha} - \boldsymbol{\beta} \in \mathbf{nul}(\Phi) \cap \mathbb{N}^{k+r+1}.
\]
Consequently,
\[
\overline{F}(\mathbf{x}) = y_1 \cdots y_k z_1 \cdots z_r x^{s+1} F(\mathbf{x}).
\]
By Stanley's reciprocity theorem (Theorem~\ref{StanleyReciprocity}),
\[
x^{s+1} F(\mathbf{1}, x) = \overline{F}(\mathbf{1}, x) = (-1)^{k+1} F(\mathbf{1}, x^{-1}),
\]
that is,
\begin{align}\label{EHRPGMH}
x^{s+1} \mathrm{Ehr}(P(H), x) = (-1)^{k+1} \mathrm{Ehr}(P(H), x^{-1}).
\end{align}

Equation~\eqref{EHRPGMH} is therefore equivalent to
\[
M(x) \cdot x^{s+1} = (-1)^{k+1} M(x^{-1}) \frac{Q(x)}{Q(x^{-1})}.
\]
By Lemma~\ref{Ehrhartpoled+1} and $\dim P(H)=k$, we have $M(x) = x^{\deg Q(x) - s - 1} M(x^{-1})$.
Thus, the numerator $M(x)$ is a symmetric polynomial of degree $\deg Q(x) - s - 1$.
This completes the proof.
\end{proof}

Combining Theorems~\ref{UnimodularHyperG} and~\ref{UniformHyperGr}, we immediately obtain the following corollary.
\begin{cor}\label{cor:ehrhart-series}
Let $H = (V, E^*)$ be a finite simple connected hypergraph with vertex set $V = \{v_1, \dots, v_k\}$ and edge set $E^* = (e_1, \dots, e_r)$ such that $\bigcup_{i=1}^r e_i = V$. If $H$ is $s$-uniform and unimodular, then the Ehrhart series of $P(H)$ has the rational form
\[
\mathrm{Ehr}(P(H), x) = \frac{M(x)}{(1-x)^{k+1}},
\]
where $M(x)$ is a symmetric polynomial of degree $k - s$.
\end{cor}

In the following examples, we compute the Ehrhart series using the software packages \texttt{LattE}~\cite{DeLoera04} and \texttt{CTEuclid}~\cite{Xin15}; the results agree.

\begin{figure}[htp]
\centering
\includegraphics[width=8cm,height=3cm]{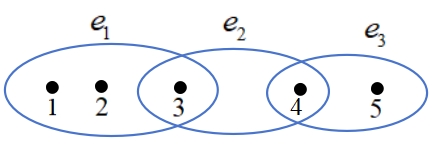}
\caption{A unimodular hypergraph with $3$ hyperedges}
\label{Figure1}
\end{figure}

\begin{exa}
Consider the hypergraph $H = (V, E^*)$ where $E^* = (e_1, e_2, e_3)$, as illustrated in Figure~\ref{Figure1}. The hyperedges are given by $e_1 = \{1,2,3\}$, $e_2 = \{3,4\}$, and $e_3 = \{4,5\}$.
The incidence matrix of $H$ is the $3 \times 5$ matrix
\[
A = \begin{pmatrix}
1 & 1 & 1 & 0 & 0 \\
0 & 0 & 1 & 1 & 0 \\
0 & 0 & 0 & 1 & 1
\end{pmatrix}.
\]
This matrix is totally unimodular; hence $H$ is unimodular. Note that $H$ is not uniform. Furthermore, the Ehrhart series of $P(H)$ is
\[
\mathrm{Ehr}(P(H),x) = \frac{1+6x+4x^2}{(1-x)^6}.
\]
\end{exa}

\begin{figure}[htp]
\centering
\includegraphics[width=8cm,height=7cm]{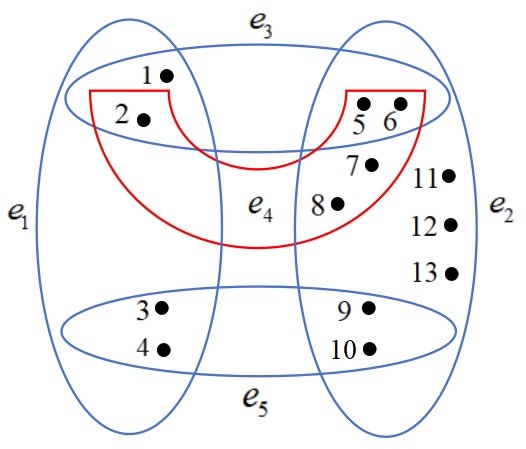}
\caption{A unimodular hypergraph with $5$ hyperedges}
\label{Figure2}
\end{figure}

\begin{exa}
Consider the hypergraph $H = (V, E^* = (e_1, \dots, e_5))$ shown in Figure~\ref{Figure2}, where
\begin{align*}
e_1 &= \{1,2,3,4\}, & e_2 &= \{5,6,7,8,9,10,11,12,13\}, \\
e_3 &= \{1,2,5,6\}, & e_4 &= \{2,5,6,7,8\}, \qquad\qquad
e_5 = \{3,4,9,10\}.
\end{align*}
The incidence matrix of $H$ is the $5 \times 13$ matrix
\[
A=\left(
\begin{array}{ccccccccccccc}
1 & 1 & 1 & 1 & 0 & 0 & 0 & 0 & 0 & 0 & 0 & 0 & 0  \\
0 & 0 & 0 & 0 & 1 & 1 & 1 & 1 & 1 & 1 & 1 & 1 & 1  \\
1 & 1 & 0 & 0 & 1 & 1 & 0 & 0 & 0 & 0 & 0 & 0 & 0  \\
0 & 1 & 0 & 0 & 1 & 1 & 1 & 1 & 0 & 0 & 0 & 0 & 0  \\
0 & 0 & 1 & 1 & 0 & 0 & 0 & 0 & 1 & 1 & 0 & 0 & 0  \\
\end{array}
\right).
\]
This matrix is totally unimodular, confirming that $H$ is unimodular. Note that $H$ is not uniform. The Ehrhart series of the polytope $P(H)$ is
\[
\mathrm{Ehr}(P(H),x) = \frac{1+26x+167x^2+300x^3+125x^4}{(1-x)^{14}}.
\]
\end{exa}

\begin{figure}[htp]
\centering
\includegraphics[width=7cm,height=5cm]{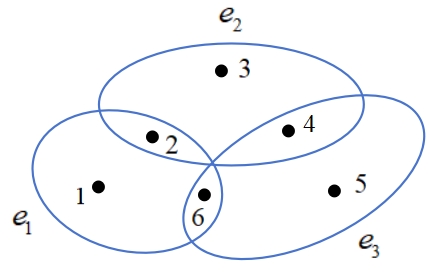}
\caption{A $3$-uniform hypergraph with $3$ hyperedges}
\label{Figure3}
\end{figure}

\begin{exa}
Let $H = (V, E^* = (e_1, e_2, e_3))$ be the hypergraph shown in Figure~\ref{Figure3}, where $e_1 = \{1,2,6\}$, $e_2 = \{2,3,4\}$, and $e_3 = \{4,5,6\}$.
The incidence matrix of $H$ is the $3 \times 6$ matrix
\[
A = \begin{pmatrix}
1 & 1 & 0 & 0 & 0 & 1 \\
0 & 1 & 1 & 1 & 0 & 0 \\
0 & 0 & 0 & 1 & 1 & 1
\end{pmatrix}.
\]
Clearly, $H$ is $3$-uniform. However, $A$ is not totally unimodular. The Ehrhart series of $P(H)$ is
\[
\mathrm{Ehr}(P(H), x) = \frac{1 + 8x + 15x^2 + 8x^3 + x^4}{(1-x)^6 (1-x^2)}.
\]
\end{exa}

\begin{figure}[htp]
\centering
\includegraphics[width=10cm, height=6cm]{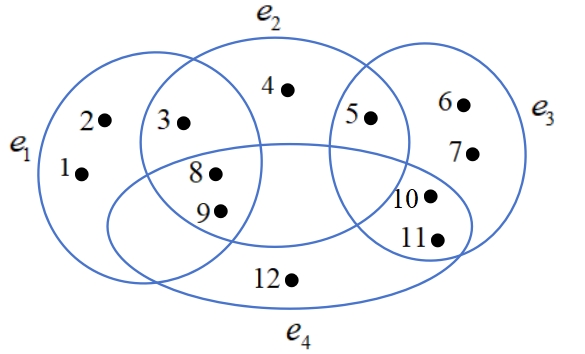}
\caption{A $5$-uniform hypergraph with $4$ hyperedges}
\label{Figure4}
\end{figure}

\begin{exa}
Consider the hypergraph $H = (V, E^* = (e_1, e_2, e_3, e_4))$ shown in Figure~\ref{Figure4}, where $e_1 = \{1,2,3,8,9\}$, $e_2 = \{3,4,5,8,9\}$, $e_3 = \{5,6,7,10,11\}$, and $e_4 = \{8,9,10,11,12\}$.
The incidence matrix of $H$ is the $4 \times 12$ matrix
\[
A=\left(
\begin{array}{cccccccccccc}
1 & 1 & 1 & 0 & 0 & 0 & 0 & 1 & 1 & 0 & 0 & 0 \\
0 & 0 & 1 & 1 & 1 & 0 & 0 & 1 & 1 & 0 & 0 & 0 \\
0 & 0 & 0 & 0 & 1 & 1 & 1 & 0 & 0 & 1 & 1 & 0 \\
0 & 0 & 0 & 0 & 0 & 0 & 0 & 1 & 1 & 1 & 1 & 1
\end{array}
\right).
\]
Clearly, $H$ is $5$-uniform. A straightforward verification shows that $A$ is not totally unimodular. The Ehrhart series of $P(H)$ is
\begin{align*}
\mathrm{Ehr}(P(H),x)&=\frac{1+60x+835x^2+4908x^3+15658x^4+30334x^5+37628x^6}{(1-x)^{13}(1+x)^5}
\\ &\quad\quad + \frac{30334x^7+15658x^8+4908x^9+835x^{10}+60x^{11}+x^{12}}{(1-x)^{13}(1+x)^5}.
\end{align*}
\end{exa}

\section{Concluding Remark}

In this paper, we establish Conjectures \ref{ConjeConnGrap1} and \ref{ConjeCircularGrap3} through an application of Stanley's reciprocity theorem.
Recall that a polytope $\mathcal{P}$ is \emph{reflexive} if it is integral and admits a representation
\[
\mathcal{P} = \{\mathbf{x} \in \mathbb{R}^d \mid A\mathbf{x} \leq \mathbf{1}\},
\]
where $A$ is an integer matrix.
For such polytopes, the symmetry of the $h^*$-polynomial is characterized by Hibi's palindromic theorem \cite[Theorem 4.6]{BeckRobins}.
In our setting, the simple $s$-uniform hypergraph polytope $P(H)$ constitutes a rational polytope.
Theorem~\ref{UniformHyperGr} characterizes the symmetry of the numerator polynomial in the Ehrhart series $\mathrm{Ehr}(P(H), x)$.

{\small
\textbf{Acknowledgments:}
The author would like to thank the anonymous referee for valuable suggestions that improved the presentation.
He is grateful to Yueming Zhong for insightful discussions and to his advisor Guoce Xin for guidance and support.
This work was partially supported by the National Natural Science Foundation of China [12071311] and [12571355].
}

%
%

\end{document}